\documentclass[11pt]{amsart} 
\usepackage{amsmath, amsthm,url,
  amssymb,setspace,epsfig, hyperref, mathrsfs,fontenc, yhmath,
  fullpage, MnSymbol}

\usepackage[all]{xy}

\newtheorem{theorem}{Theorem}
\newtheorem{proposition}[theorem]{Proposition}
\newtheorem{corollary}[theorem]{Corollary}
\newtheorem{lemma}[theorem]{Lemma}

\newtheorem{prob}[theorem]{Problem}

\theoremstyle{definition}
\newtheorem{definition}[theorem]{Definition}
\theoremstyle{remark}
\newtheorem{exam}[theorem]{Example}

\newtheorem{cond}[theorem]{Condition}

\newtheorem{ntn}[theorem]{Convention}

\newtheorem{remark}[theorem]{Remark}

\numberwithin{equation}{section}

\newcommand         {\rar}[1]       {\stackrel{#1}{\longrightarrow}}
\newcommand         {\isom}         {\rar{\simeq}}
\newcommand         {\commentout}[1]    {}

\def    \ra     {\rightarrow}
\def    \inj    {\hookrightarrow}
\def    \onto   {\twoheadrightarrow}
\def    \*      {\times}

\def    \cO     {\mathcal{O}}
\def	\cOt	{\cO^\times}

\def	\fp		{\mathfrak{p}}

\def	\fm		{\mathfrak{m}}

\def    \cond   {\mathrm{cond}}

\def	\Hom	{\operatorname{Hom}}

\def    \SL     {\mathrm{SL}}
\def    \ind    {\operatorname{ind}}

\def    \End    {\operatorname{End}}

\def	\GL		{\mathrm{GL}}
\def	\Sp		{\mathrm{Sp}}

\def    \diag   {\mathrm{diag}}

\def    \sH    {\mathscr{H}}

\def	\sW		{\mathscr{W}}

\def 	\bbG		{\mathbb{G}}

\def    \bZ     {\mathbb{Z}}
\def    \bR     {\mathbb{R}}

\def    \bC     {\mathbb{C}}

\def    \Gm     {\mathbb{G}_m}

\def    \bCt    {\bC^\times}

\def    \Fq     {\mathbb{F}_q}

\def	  \tH		{\tilde{H}}

\def    \tG     {\tilde{G}}

\def    \hT    {\check{T}}

\newcommand{\Stab}{\operatorname{Stab}}

\newcommand{\hG}{\check{G}}

\newcommand		{\bmu}{{\bar{\mu}}}

\newcommand		\sR		{\mathscr{R}}

\newcommand	\charr		{\mathrm{char}}
\newcommand		\sV		{\mathscr{V}}

\newcommand 		\sm		{\mathrm{sm}}

\newcommand		\tQ		{(Q^\vee)_\sat}
\newcommand		\tT		{\tilde{T}}
\newcommand		\F		{\mathrm{F}}
\newcommand		\G		{\mathrm{G}}
\newcommand		\sat		{\mathrm{sat}}

\newcommand \tchi {\tilde{\chi}}

\usepackage{ifpdf}
\ifpdf
\else
\PassOptionsToPackage{dvistyle}{todonotes}
\fi
\usepackage[
color=orange!80, 
bordercolor=black,
textwidth=.8in,
textsize=small]
{todonotes}

\begin{document}
\title{Ramified Satake isomorphisms for strongly parabolic characters}

\author{Masoud Kamgarpour} 
\address{Hausdorff Center for Mathematics, Endenicher Allee 62, 
53115 Bonn, Germany}
\email{masoudkomi@gmail.com}

\author{Travis Schedler}
\address{University of Texas at Austin, Mathematics Department, RLM 8.100, 2515 Speedway Stop C1200, Austin, TX 78712, USA}
\email{trasched@gmail.com}

\thanks{M.K. was supported by a Hausdorff Institute Fellowship.  T.S. was supported by an AIM Fellowship and the ARRA-funded NSF grant DMS-0900233.}


\begin{abstract} For certain characters of the compact torus of a reductive $p$-adic group, which we call strongly parabolic characters, we prove Satake-type isomorphisms. Our results generalize those of Satake, Howe, Bushnell and Kutzko, and Roche.
 \end{abstract} 

\maketitle

 \tableofcontents

\section{Introduction}
\subsection{The problem we study} 
Let $F$ be a local non-Archimedean field with ring of integers $\cO$
and residue field $\Fq$. Let $G$ be a connected split reductive group
over $F$ with split torus $T$ and Weyl group $W=N_G(T)/T$. Let $\hT$
denote the dual torus. Replacing $G$ by an isomorphic group, we may
assume that $G$ is defined over $\bZ$. Then $G(\cO)$ is a maximal compact (open) subgroup of
$G(F)$. Let $\sH(G(F), G(\cO))$ denote the convolution algebra of
compactly supported $G(\cO)$-bi-invariant complex valued functions on
$G(F)$. A celebrated theorem of Satake \cite{Satake} states that we
have a canonical isomorphism of algebras
\begin{equation}\label{eq:classicalSatake}
\sH(G(F), G(\cO)) \simeq \bC[\hT/W].
\end{equation} 
We are interested in generalizing this isomorphism to nontrivial
smooth characters $\bmu: T(\cO)\ra \bCt$, as follows.  Let
$W_{\bmu}\subseteq W$ denote the stabilizer of $\bmu$ under the action
of the Weyl group.  Then it is natural to pose:
\begin{prob}\label{p:main}
Construct a pair $(K,\mu)$ consisting of a compact open subgroup $T(\cO)\subseteq K\subseteq G(\cO)$ and a character $\mu:K\ra \bCt$ extending $\bmu$, such that we have an isomorphism of algebras 
\begin{equation} \sH(G(F), K,\mu)\simeq \bC[\hT/W_{\bmu}],
\end{equation}
where $\sH$ is the convolution algebra of $(K,\mu)$-bi-invariant compactly supported functions on $G(F)$.
\end{prob} 

The Satake isomorphism provides a solution for the above problem for
$\bmu=1$. In this paper, we solve the above problem for a large class
of characters of $T(\cO)$ which we call ``strongly parabolic
characters,'' which are by definition characters such that $W_\bmu$ is
the Weyl group of a Levi subgroup $L < G$, and moreover such that
$\bmu$ extends to $L(F)$. This appears to be the proper generality
where the problem has a positive solution. Our construction of $K$ is
tied to $L$.  We think of the isomorphism $\sH(G(F),K,\mu)\simeq
\bC[\hT/W_\bmu]$ as a Satake isomorphism for the (possibly) ramified
character $\bmu$. Therefore, we call these isomorphisms \emph{ramified
  Satake isomorphisms}.  For characters that are not strongly
parabolic, we do not have a reason to expect a positive answer to
Problem \ref{p:main}.

\subsection{History}\label{ss:history}
Following Satake, R. Howe studied Problem \ref{p:main} for $G=\GL_N$
\cite{Howe73}.  Via an isomorphism which he called the
$\bmu$-spherical Fourier transform, he completely solved the problem
for the general linear group. Howe's paper went largely unnoticed;
however, several cases of Problem \ref{p:main} were subsequently
solved using other methods.

In \cite{Bernstein84}, \cite{Bernstein92}, Bernstein constructed a
decomposition of the category of representations of $G(F)$ using the
theory of Bernstein center. Each block admits a projective
generator. In particular, for every character $\bmu:T(\cO)\ra \bCt$,
one has a block of representations of $G(F)$, which we denote by
$\sR_\bmu(G)$. Bernstein proved that the center of $\sR_\bmu(G)$ is
canonically isomorphic to $\bC[\hT/W_\bmu]$; see, for instance,
\cite[Theorem 1.9.1.1]{Roche09}. Moreover, he gave an explicit
description of a projective generator for each of these blocks; see
the RHS of \eqref{eq:IntroPhi}. When the character $\bmu$ is
\emph{regular}; i.e., $W_\bmu=\{1\}$, then the center is $\bC[\hT]$, and it
identifies canonically with the endomorphism ring of Bernstein's
generator.

In a fundamental paper \cite{Bushnell98}, Bushnell and Kutzko
organized the study of representations of $G(F)$ via compact open
subgroups into the theory of \emph{types}. Namely, they proposed that
one should be able to obtain a projective generator for every block of
representations of $G(F)$ by inducing a finite dimensional
representation from a compact open subgroup. The pair of the compact
open subgroup and its finite dimensional representation, up to a
certain equivalence, is called the type.  In \cite{Bushnell99} and
\cite{BKGL}, they explicitly construct types for every block of
representations of $\GL_N$. In particular, they construct projective
generators for the principal series blocks $\sR_\bmu(\GL_N)$. When the
character $\bmu$ is regular, their construction provides a pair
$(K,\mu)$ satisfying the requirement of Problem \ref{p:main}. We note,
however, that Bushnell and Cuzco's construction of types is
technically involved, since they consider all blocks (not merely the
principal series blocks); in particular, we were not able to locate
exactly where in their papers they construct types for the principal
series blocks of $\GL_N$.

Finally, Roche \cite{Roche98} constructed types for principal series
representations of arbitrary reductive groups in good characteristics
(which excluded in particular those listed in Convention
\ref{n:char}). In the case that $\bmu$ is regular, the type itself is
a pair $(K,\mu)$ satisfying the conditions of Problem \ref{p:main}.

In this paper, we build on the methods introduced by Bushnell and
Kutzko and Roche, and solve the problem for all strongly parabolic
characters. We make use of Roche's type in order to construct a pair
$(K,\mu)$ satisfying the conditions of Problem \ref{p:main}.

\subsection{On characters of $T(\cO)$} A significant part of this
paper, which may be of independent interest, is devoted to defining
and studying certain smooth characters of $T(\cO)$. Recall that a
subgroup $W'\subseteq W$ is \emph{parabolic} if it is generated by
simple reflections. The Levi subgroup $L$ associated to $W'$ is the
subgroup generated by $T$ and the the simple roots corresponding to
the simple reflections in $W'$ along with their negatives.

\begin{definition} Let $\bmu:T(\cO)\ra \bCt$ be a smooth character. 
\begin{enumerate} 
\item[(i)]  $\bmu$ is \emph{parabolic} if
the stabilizer $\Stab_W(\bmu)$ of $\bmu$ in $W$ is a parabolic subgroup.  
\item[(ii)] $\bmu$ is \emph{strongly parabolic} if it is parabolic
  with Levi $L$ and extends to a character of $L(F)$.
\item[(iii)] $\bmu$ is \emph{easy} it is parabolic and it extends to a
  character of $L(F)$ which is trivial on $[L,L](F)$.
\end{enumerate}
\end{definition}  
It follows immediately from the definition that the trivial character
and all regular characters are easy. Moreover, it is clear that
\begin{equation} 
  \textrm{easy} \implies \textrm{strongly parabolic}
 \implies \textrm{parabolic}. 
\end{equation} 
The reverse implications can all fail; see Examples
\ref{ex:strongpar-not-easy} and \ref{ex:stronglyPar}. 

To state our results regarding these characters, we need some notation. 
Let $\Delta$ denote the set of roots of $G$. 
Let ${X}$, ${X^\vee}$, $Q$, $Q^\vee$
denote the character, cocharacter, root and coroot lattices of $G$,
respectively. Below we will frequently impose the conditions that either ${X}/Q$
is free or ${X^\vee}/Q^\vee$ is free (or both).  We remark that
$X^\vee/Q^\vee$ being free is equivalent to $[G(\bC),G(\bC)]$ being
simply-connected, while $X/Q$ being free is equivalent to the statement that
$G(\bC)$ has connected center.\footnote{This follows from the fact that if $G$ is a (connected split) semisimple group, then $X/Q$ equals the dual of $Z(G(\bC))$ and $X^\vee/Q^\vee$ equals the dual of $\pi_1(G(\bC))$; see, for example \cite[Example 6.7]{Conrad}. For example,  for $\SL_2$, we have $(X,Q,X^\vee,Q^\vee)=(\bZ, 2\bZ, \bZ, \bZ)$.}

\begin{theorem} \label{t:characters} 
Let $\bmu:T(\cO)\ra \bCt$ be a smooth character.
\begin{enumerate} 
\item[(i)] $\bmu$ is easy if and only if it is parabolic and can be
  written as a product $\chi_1\cdots \chi_l$, where each $\chi_i$ is a
  character $T(\cO)\ra \bCt$ which is a composition of a
  $W_\bmu$-invariant rational character $T(\cO)\ra \cOt$ and a smooth character
  $\cOt\ra \bCt$.
\item[(ii)] The following are equivalent:
\begin{enumerate}
\item[(a)] $\bmu$ is strongly parabolic;
\item[(b)] $\bmu\circ \alpha^\vee|_{\cOt}=1, \quad \forall \alpha \in \Delta_L$.
\end{enumerate}
Moreover, if $q > 2$, then these are also equivalent to:
\begin{enumerate}
\item[(c)] $\bmu$ extends to a character of $L(\cO)$.
\end{enumerate}

\item[(iii)] If $X/Q$ is free or $\Delta$ has no
  factors of type $A_1$ or $C_n$, then every parabolic character of
  $T(\cO)$ is strongly parabolic.

\item[(iv)] If $X^\vee/Q^\vee$ is free, then every strongly parabolic
  character of $T(\cO)$ is easy.

\item[(v)] If $\Delta$ is simply-laced and $X/Q$ is free, then
  every character of $T(\cO)$ is strongly parabolic. 
  \end{enumerate} 
\end{theorem}
Section \ref{s:characters} is devoted to the proof of the above theorem.

We now indicate what the above theorem implies for characters of various groups. By $G/Z$ we mean $G/Z(G)$. The letter $N$ denotes a positive integer. We let $\mathrm{E}_n$, $n=6,7,8$ (resp. $\F_4$ and $\G_2$) denote the split reductive group whose
associated complex group is the connected, simply-connected, simple
group of type $\mathrm{E}_n$ (resp. $\F_4$ and $\G_2$).   
\begin{equation}\label{t:groups}
 \begin{tabular}{|c|c|c|}  \hline
 {\textbf{Reductive group}} &  {\textbf{Properties}} &  {\textbf{Characters}} \\ \hline
$\GL_N$,  $\mathrm{E}_8$ & simply-laced, ${X}/Q$ and ${X^\vee}/Q^\vee$ free & 
{\textrm{all characters are easy}} \\\hline
$\mathrm{PGL}_N$, $\mathrm{GO}_{2N},$ & \textrm{simply-laced and} & 
\textrm{all characters are} \\
$\mathrm{SO}_{2N}/Z$, $\mathrm{E}_6/Z, \mathrm{E}_7/Z$ &   X/Q free & {\textrm{strongly parabolic}}\\ \hline
$\mathrm{SL}_N$ ($N \geq 3$), $\mathrm{GSp}_{2N}$, & $X^\vee/Q^\vee$ free, and hypothesis of (iii)
  &  {\textrm{all parabolic characters are easy}} \\
 $\mathrm{Spin}_N, E_N$ ($N \geq 6$), $\mathrm{F}_4$, $\mathrm{G}_2$ & &
\\
\hline
$\Sp_{2N}/Z$, $\mathrm{GO}_N$, $\mathrm{SO}_N$
&  hypothesis of (iii)
& \textrm{all parabolic characters} 
\\ & & \textrm{are strongly parabolic}
 \\ \hline
\end{tabular}
\end{equation}

\begin{remark} Let $G$ be a (connected) algebraic group over a field
  $k$. Let $\bar{k}$ denote an algebraic closure of $k$. Then $G$ is
  said to be \emph{easy} if every $g\in G(\bar{k})$ is in the neutral
  connected component of its centralizer in $G\otimes_k \bar{k}$. This
  definition is due to V. Drinfeld. Based on the discussion in, e.g.,
  \cite[\S 2.2]{Boy-cugff}, there appears to be a relationship between
  Drinfeld's notion of easy and ours, when $k$ has characteristic zero. Namely, here we show that, if $[G,G]$ is simply connected
  and $Z(G)$ is connected, then every parabolic character is easy (and
  the parabolic assumption is not needed in the simply-laced case); in
  \cite[\S 2.2]{Boy-cugff} it is asserted, without proof, that these
  two assumptions are equivalent (over a field of characteristic zero) to $G$ being easy in
  Drinfeld's sense.
\end{remark}

\begin{remark} \label{r:roche-th} To every character $\bmu:T(\cO)\ra
  \bCt$, Roche \cite[\S 8]{Roche98} associated a possibly disconnected
  split reductive group $\tH=\tH_\bmu$ over $F$. The connected
  component of $\tH$ is an endoscopy group for $G$. It follows from
  Theorem \ref{t:characters}.(ii) that strongly parabolic characters
  are exactly those characters for which $\tH$ is the Levi of a
  parabolic of $G$ (and in particular connected).  In more detail, by
  \cite[Definition 6.1]{Roche98}, the coroots $\alpha^\vee$ of the
  connected component $H$ of the identity of $\tH$ (as a complex
  reductive group) are exactly those for which $\bmu \circ \alpha^\vee
  |_{\cOt} = 1$, and by \cite[Lemma 8.1.(i)]{Roche98}, the stabilizer
  of $\bmu$ equals the Weyl group of $H$ (and is not bigger) if and
  only if $\tH = H$. Then, we conclude because the Weyl group of $H$
  is a parabolic subgroup of the Weyl group of $G$
  if and only if $H$ is a Levi subgroup of $G$ (i.e., its roots form a
  closed root subsystem of those of $G$).
\end{remark}

\subsection{Satake isomorphisms} 
In this section, we let $G$ be a connected split reductive group over a local field $F$. 
We impose the following restrictions on the residue characteristic of $F$. 
\begin{ntn} \label{n:char} For every irreducible direct factor of the
  root system of $G$, we assume that the residue characteristic of $F$
is not one of the following primes:
\begin{equation}\label{eq:char}
  \begin{tabular}{|c|c|}  \hline
    Root system & Excluded primes \\
    \hline
    $B_n, C_n, D_n$ & $\{2\}$ \\ \hline
    $F_4, G_2, E_6, E_7$ & \{2,3\} \\ \hline
    $E_8$ & $\{2,3,5\}$ \\ \hline
\end{tabular}
\end{equation}
\end{ntn}

\begin{theorem}\label{t:Satake} Let $G$ be a connected split reductive 
  group over a local field $F$ whose residue characteristic satisfies
  the above restrictions.  Then for every strongly parabolic character
  $\bmu:T(\cO)\ra \bCt$, there exists a compact open subgroup
  $K<G(\cO)$ and an extension $\mu:K\ra \bCt$ such that
\begin{equation}\label{eq:Satake}
\sH(G(F), K,\mu) \simeq \bC[\hT/W_\bmu]
\end{equation} 
\end{theorem}

 As mentioned above, in the
case of $G=\GL_N$, the above theorem is due to Howe \cite{Howe73}, and  if $\bmu$ is regular, then
the above theorem follows by combining results of Bernstein \cite{Bernstein84}, \cite{Bernstein92},
Bushnell-Kutzko \cite{Bushnell98}, \cite{Bushnell99} and Roche
\cite{Roche98}. As far as we know, the generalization to strongly
parabolic characters is new.

\begin{exam} \label{ex:GLN} Let $G=\GL_3$ and let $T(\cO)\simeq
  (\cOt)^3$ denote the group of diagonal matrices. Write
  $\bmu=(\bmu_1,\bmu_2,\bmu_3)$ where each $\bmu_i$ is a smooth character $\cOt\ra \bCt$. Suppose $\bmu_1=\bmu_2$ and that the conductor
  $\cond(\bmu_1/\bmu_3)$ equals  $n\geq 2$. (The conductor of a character $\chi:\cOt\ra
  \bCt$ is the smallest positive integer $c$ for which
  $\chi(1+\fp^c)=\{1\}$.)   If we follow Howe's approach, we would take
  \[
  K= \begin{pmatrix} \cO & \cO & \cO  \\
    \cO & \cO  &\cO  \\
    \fp^{n} & \fp^{n}  & \cO
  \end{pmatrix} \cap G(\cO). 
  \]
  On the other hand, in the present article, following more closely the types of \cite{Roche98}, we take instead
  \[
  K= \begin{pmatrix} \cO & \cO & \fp^{[\frac{n}{2}]}  \\
    \cO & \cO  &\fp^{[\frac{n}{2}]}  \\
    \fp^{[\frac{n+1}{2}]} & \fp^{[\frac{n+1}{2}]}  & \cO
  \end{pmatrix} \cap G(\cO). 
  \]
In both cases, $\bmu$ extends to a character $\mu:K\ra \bCt$ and one
has an isomorphism $\sH(G(F), K,\mu)\simeq \bC[\hT]$. This example
shows that the subgroup $K$ of Theorem \ref{t:Satake} is not necessarily unique.
\end{exam}

To prove Theorem \ref{t:Satake}, we use Roche's result on types for principal series representations.
Given an arbitrary smooth character $\bmu:T(\cO)\ra \bCt$,
Roche \cite{Roche98} constructed a compact open subgroup $J\subset
G(F)$ (which depends on the choice of $B$) and an extension $\mu^J: J\ra \bCt$ such that the compactly
induced representation
\begin{equation} \sW:=\ind_J^{G(F)} \mu^J
\end{equation} 
is a progenerator for the principal series Bernstein block of $G$
defined by $\bmu$. More precisely, a combination of results of Bushnell and Kutzko, Dat, and Roche implies that in this situation, one has an explicit isomorphism of $G(F)$-modules 
\begin{equation} \label{eq:IntroPhi}
\Phi: \sW\isom \Pi:=\iota_{B(F)}^{G(F)} \left( \ind_{T(\cO)}^{T(F)} \bmu \right).
\end{equation} 
Here, $\iota$ denotes the functor of parabolic
induction.\footnote{Note that here and in \eqref{eq:Theta}, it does not matter if we use normalized
  or unnormalized parabolic induction since the representation being
  induced is isomorphic to its twist by any unramified character.}
See \S \ref{ss:main2} for the explicit description of $\Phi$. 
Note that the endomorphism algebra of $\sW$ is canonically
isomorphic with $\sH(G(F), J, \mu^J)$. 

Now suppose the character $\bmu$ is strongly parabolic. Let $L$ denote
the corresponding Levi and let $\mu^{L(F)}:L(F)\ra \bCt$ denote an
extension of $\bmu$ to $L(F)$.  Let $\mu^L = \mu^{L(\cO)} :=
\mu^{L(F)}|_{L(\cO)}$ denote its restriction to $L(\cO)$. We prove
that $K=JL(\cO)$ is a subgroup of $G(F)$. Moreover, we show that there
exists a canonical character $\mu:K\ra \bCt$ which extends $\mu^J$ and
$\mu^L$. Theorem \ref{t:Satake} states that the Hecke algebra
$\sH(G(F), K,\mu)$, consisting of compactly supported
$(K,\mu)$-bi-invariant functions on $G(F)$, is isomorphic to
$\bC[\hT/W_\bmu]$. To prove this result, we realize $\sH(G(F), K,\mu)$
as an endomorphism ring of a family of principal series
representations, which we call a \emph{central family}.

\subsection{Central families}
\begin{definition} \label{d:central}
Let $\bmu$ be a strongly parabolic character with the corresponding Levi $L$. Let $K=JL(\cO)$ denote the corresponding compact open subgroup. The \emph{central family} of principal series representations of $G$ attached to $\bmu$ is defined by
\begin{equation} \label{e:sv-defn}
\sV:=\ind_K^{G(F)} \mu. 
\end{equation} 
\end{definition} 

Note that $\sV$ is a submodule of $\sW$ and the latter is a
progenerator for the principal series block corresponding to
$\bmu$. According to Theorem \ref{t:Satake}, the endomorphism ring
$\sH$ of this family identifies with the center of the corresponding
Bernstein block (which is isomorphic to the center of $\sH(G(F),
J,\mu^J)$, and hence isomorphic to
$\bC[\hT/W_\bmu]$; cf.~\S \ref{ss:history}). Moreover, one can show
that, for generic maximal ideals $\fm\subset \sH$, the $G(F)$-module
$\sV/{\fm \sV}$ is an irreducible principal series representation.
(We will neither prove nor use the last statement.) We will now give
an alternative description of $\sV$. Let $P\supseteq B$ be a parabolic
subgroup whose Levi is isomorphic to $L$. Let

\begin{equation} \label{eq:Theta}
\Theta:=\iota_{P(F)}^{G(F)} \left(\ind_{L(\cO)}^{L(F)} \mu^L\right).
\end{equation}

\begin{theorem} \label{t:main2} Under the assumptions of Theorem
  \ref{t:Satake}, we have a canonical isomorphism of $G(F)$-modules
  $\sV \isom \Theta$.
\end{theorem}

We prove the above theorem by identifying $\sV$ and $\Theta$ with submodules of $\sW$ and $\Pi$, respectively. Then, using the explicit description of $\Phi$ in \eqref{eq:IntroPhi}, we show that $\Phi\, |_\sV: \sV\ra \Pi$ defines an isomorphism onto $\Theta$. On the other hand, the endomorphism ring $\sV$ identifies with $\sH(G(F),
K,\mu)$. Thus, to prove Theorem \ref{t:Satake}, we need to compute the endomorphism algebra of $\Theta$. To this end, we will use a
theorem of Roche \cite{Roche02} on parabolic induction of Bernstein
blocks. 

\begin{remark} 
\begin{enumerate} 
\item[(i)] As mentioned above, in this paper, we construct the pair $(K,\mu)$ satisfying requirement of Problem \ref{p:main} by using Roche's pair $(J,\mu^J)$. In this case, the subgroup $K$ depends only on the kernel of $\bmu$; that is, if $\ker(\bmu)=\ker(\bmu')$ then $K_{\bmu}=K_{\bmu'}$. In fact, it only depends
on the conductors of the restrictions of $\bmu$ to the coroot subgroups (i.e., the minimal $c_\alpha \geq 1$ such that $\bmu|_{\alpha^\vee(1+\fp^{c_\alpha})}$ is trivial) together with the collection of roots $\alpha$ such that the entire
restriction $\bmu|_{\alpha^\vee(\cOt)}$ is trivial.
This follows immediately from the construction of $J$; see \S \ref{s:K}. 

\item[(ii)] The pair $(J,\mu^J)$ is a \emph{type} for the Bernstein
  block $\sR_\bmu(G)$. Types for Bernstein blocks are not, however,
  necessarily unique. Therefore, it is natural to wonder if our
  construction could work using a different type $(J',\mu^{J'})$. In
  the case $G = \GL_N$, this is true in view of the results of
  \cite{Howe73}, as we observed (for $N=3$) in Example
  \ref{ex:GLN}. We do not, however, pursue this question in the
  current text.
\end{enumerate} 
\end{remark}

\subsection{Further directions} \label{ss:directions}
The proof of Theorem \ref{t:Satake} given in this paper is rather
indirect; moreover, it relies on nontrivial results of Bernstein,
Bushnell and Kutzko, Roche, and Dat.  In a forthcoming paper
\cite{ramified2}, we hope to give a direct proof of this theorem by
writing an explicit support preserving isomorphism $\sH(L(F), L(\cO),
\mu^L) \isom \sH(G(F), K,\mu)$. In other words, we hope to prove  Theorem \ref{t:Satake} using combinatorics and the classical Satake isomorphism. This proof should also make clear
the geometric nature of the group $K$ and some of its double cosets in
$G(F)$; in particular, we expect that it will help with the
geometrization program (see below).

In Definition \ref{d:central}, for strongly parabolic characters of the compact torus, we constructed ``central families". The endomorphism ring of the central family identifies canonically with the center of the block defined by the character; moreover, generic irreducible representations in the block appear with multiplicity one in the central family. It would be interesting to find analogous central families for other Bernstein blocks.   

It is well-known that the Satake isomorphism allows one to realize the local unramified Langlands correspondence. In more detail, 
let $\hG$ denote the complex reductive group which is the Langlands
dual of $G$. Using the classical Satake isomorphism \eqref{eq:classicalSatake}, one can show that we have a
bijection
\[
\boxed{\textrm{unramified irreducible representations of $G(F)$}} \leftrightarrow
\boxed{\textrm{characters of $\sH$}}\]
Combining this with the bijections 
\[
\boxed{\textrm{characters of $\sH$}}
 \leftrightarrow
\boxed{\textrm{points of $\hT/W$}} \leftrightarrow
\boxed{\textrm{semisimple conjugacy classes in $\hG$}}
\] 
we obtain a bijection between unramified representations of $G(F)$ and semisimple conjugacy classes in $\hG$. It would be interesting to
study the role of the ramified Satake isomorphisms (i.e., the ones
given by Theorem \ref{t:Satake}) in the local Langlands program.

In \cite{Haines10}, a version of the Satake isomorphism for non-split
groups is proved. On the other hand, there is also now a Satake
isomorphism in characteristic $p$; see \cite{Herzig}. 
We expect that there is also a version of Theorem
\ref{t:Satake} for non-split groups and one in characteristic $p$.

Finally, we expect that there is a geometric version of Theorem
\ref{t:Satake}. The geometric version of the usual Satake isomorphism
is proved by Mirkovic and Vilonen \cite{MV}, completing a project
initiated by Lusztig, Beilinson and Drinfeld, and Ginzburg. In the
case of regular characters; i.e., in the case that the stabilizer of
the character in the Weyl group is trivial, a geometric version of
Theorem \ref{t:Satake} is proved in \cite{geometrization}. In
\cite[\S 1.4]{geometrization}, we conjectured the theorems proved in this
article; moreover, we formulated precise conjectures for geometrizing
these results. We hope to return to this theme in future work.

\subsection{Acknowledgements} 
We would like to thank Alan Roche for very helpful email
correspondence. He sketched proofs of several technical results for
us; moreover, he brought to our attention the references
\cite{Roche02}, \cite{Dat99}, and \cite{Blondel}.  We would like to
thank J. Adler for reading an earlier draft and several useful
discussions.  We also thank Loren Spice for helpful conversations. We
thank V. Drinfeld for pointing out to us Proposition
\ref{p:drinfeldMorphism} which plays a crucial role in this paper and
sharing with us his notes on easy algebraic groups. Finally, we thank
the Max Planck Institute for Mathematics in Bonn for its hospitality.


\section{Parabolic, strongly parabolic, and easy characters} 
\label{s:characters}
\subsection{Conventions}\label{ss:conventions}
Let $F$ be a local field with ring of integers $\cO$, unique maximal
ideal $\fp$, residue field $\Fq$, and uniformizer $t$. Let $G$ be a
connected split reductive group over $F$ with $F$-split torus
$T$. Replacing $G$ if necessary by an $F$-isomorphic group, we may
(and do) assume that $G$ and $T$ are defined over $\bZ$.
Let $W=N_G(T)/T$ denote the Weyl group. 

Let $\Delta=\Delta_G$  denote the roots of $G$ (with respect to $T$). For
$\alpha$ a root in $\Delta$, we write $\alpha^\vee$ for the
corresponding coroot.  Let  ${X} =
\Hom(T,\bbG_m)$ and ${X^\vee} = \Hom(\Gm,T)$  denote the the character and cocharacter lattices,
respectively. Let $Q \subseteq {X}$ be the root lattice, and let
$Q^\vee \subseteq {X^\vee}$ denote the coroot lattice.  Let $\tQ$ be the saturation of $Q^\vee$ in
  ${X^\vee}$, i.e., 
  \[
  \tQ = \{\lambda \in X^\vee \mid m \cdot
  \lambda \in Q^\vee, \text{some } m \in \bZ\}.
  \] By definition ${X^\vee}/\tQ$ is a torsion free abelian group. To an element
$\lambda \in {X^\vee}$, we associate
$t^\lambda = \lambda(t) \in T(F)$. 

For every $\alpha \in \Delta$, let $u_\alpha: \bbG_a \to G$ be the
one-parameter root subgroup, where $\bbG_a$ is the additive group. We
assume these root subgroups satisfy the conditions specified in
\cite[\S 2]{Roche98}.  Let $U_\alpha < G$ be the image of
$u_\alpha$. For all $i \in \bZ$, let $U_{\alpha,i} = u_\alpha(\fp^i) <
G(F)$. In particular, $U_{\alpha,0}=u_\alpha(\cO)$.

Let $H$ and $K$ be topological groups and suppose $H<K$. Let
$\chi:H\ra \bCt$ be a character of $H$. We write $\ind_H^{K} \chi$ for
the space of left $(H,\chi)$-invariant relatively compactly supported
functions on $K$; that is, those functions $f:K\ra \bC$ whose support
has compact image in $K/H$ and satisfy $f(hk)=\chi(h)f(k)$ for all
$h\in H$ and $k\in K$. The group $K$ acts on this space by right
translation.

\subsection{$W$-invariant rational characters} \label{ss:Iwahori}
We start this section with a general lemma which we will repeatedly use below. 

\begin{lemma}\label{l:char-ext}
  Let $H$ be a group and $K < H$ a subgroup.  Then a character $\chi:K
  \to \bCt$ extends to a character of $H$ if and only if $\chi|_{K
    \cap [H,H]}$ is trivial.  The same is true if $H$ is an l-group
  (i.e., a locally compact totally disconnected Hausdorff topological
  group), and $K$ is a closed subgroup.\footnote{We don't need to
    assume that $H$ is totally disconnected, if we use the fact from
    \cite[Corollary 7.54]{HM-scg} that every locally compact Hausdorff
    topological group contains a compact subgroup $H'$ such that the
    quotient $H/H'$ is isomorphic to $\bR^n \times D$ for a discrete
    group $D$.}  Finally, the same is true if $H$ is a connected split
  reductive algebraic group, $K$ is a closed subgroup, and $\chi:
  K\ra \bbG_m$ is a rational character.
\end{lemma}

\begin{proof}
  It is clear that the assumption that $\chi$ be trivial on $K \cap
  [H,H]$ is necessary.  Conversely, if this is true, extending the
  character is the same as extending the induced character of $K / (K
  \cap [H,H])$ to $H/[H,H]$. Therefore, all the statements of the lemma reduce to
  the case that $H$ is commutative.

  Then, the statement that any character of a subgroup of an abstract
  (discrete) abelian group extends to the entire group follows from
  the fact that $\bCt$ is divisible, and hence injective.

For the locally compact analogue, write $\bCt \cong S^1 \times \bR_{>
  0}$. For characters to $S^1$, the statement follows from Pontryagin
duality.  For $\bR_{>0}$, note first that, if $H$ is compact, then
there are no nontrivial continuous characters to $\bR_{> 0}$. As an l-group
 always contains a compact open subgroup, this
reduces the problem to the case $H$ is discrete, where it follows as
in the previous paragraph, since $\bR_{> 0}$ is divisible, and hence
injective, as a discrete abelian group.

For the algebraic analogue, i.e., where $H$ and $K$ are connected
split tori, the statement follows because applying $\Hom(-,\bbG_m)$ to
a short exact sequence $1 \to K \to H \to H/K \to 1$ of split tori is
well-known to be an equivalence of short exact sequences of split tori
with that of their weight lattices. Hence, the restriction map from
characters of $H$ to characters of $K$ is surjective.
\end{proof}
\begin{lemma} \label{l:rationalChar} Let $G$ be a connected split
  reductive algebraic group over a field $k$ with split torus $T$. Let
  $\chi:T\ra \bbG_m$ be a rational character. The following are
  equivalent:
\begin{enumerate} 
\item $\chi$ is trivial on $T\cap [G,G]$.
\item $\chi$ extends to a character $G\ra \bbG_m$; 
\item $\chi$ is $W$-invariant; 
\item $\chi\circ \alpha^\vee: \bbG_m\ra \bbG_m$ is trivial, for every
  $\alpha \in \Delta$.
\end{enumerate} 
\end{lemma} 

\begin{proof} Lemma \ref{l:char-ext} implies immediately that $(1)\implies (2)$. Next, 
  it is clear that $[N_G(T),T]\subseteq [G,G]\cap T$; therefore, if
  we restrict a character of $G$ to $T$, we obtain a character which
  is invariant under the conjugation action of $N_T(G)$. This proves
  (2) $\implies$ (3). Next, suppose $\chi$ is $W$-invariant. Then
\begin{equation}\label{eq:roots}
  \chi\circ \alpha^\vee = (s_\alpha.\chi)\circ \alpha^\vee = 
  \chi\circ (-\alpha^\vee) = (\chi\circ \alpha)^{-1} 
\end{equation}
It follows that $(\chi\circ \alpha^\vee)^2=1$. Since $\bbG_m$ has no
nontrivial character of order $2$, it follows that $\chi\circ
\alpha^\vee=1$. Hence, (3) $\implies$ (4). For the final implication, we use the canonical identification 
  \begin{equation} \label{eq:corootgenerateT}
  T\cap [G,G] =  \bbG_m \otimes_\bZ \tQ. 
    \end{equation}
    By the notation on the RHS we mean the group subscheme of $T$
    whose $R$ points equals $R^\times \otimes_\bZ \tQ$, where $R$ is a
    ring over
    $k$.\footnote{
      For a proof of this statement over an algebraically closed field
      see, for instance, \cite{DigneMichel}, \S 0.20. Note that over
      an algebraically closed field, one does not need to use $\tQ$;
      more precisely, we have $\bar k^\times \otimes_{\bZ} Q^\vee =
      T(\bar{k})\cap [G,G](\bar{k})=(T\cap [G,G])(\bar{k})$.}  Now if
    $\chi\circ \alpha^\vee$ is trivial for every $\alpha\in \Delta$,
    then $\chi$ is trivial on $T\cap [G,G]$. This proves (4)
    $\implies$ (1).
\end{proof}

\begin{remark} \label{r:Winvariant}
It follows from the above lemma that the group of characters of $T$ which satisfy the above equivalent conditions is canonically isomorphic to 
\begin{equation}\Hom(T/(T\cap [G,G]), \bbG_m) \simeq X^W \simeq \Hom({X^\vee}/Q^\vee,\bZ) \simeq \Hom({X^\vee}/\tQ, \bZ). 
\end{equation} 
The last isomorphism follows from the following: the quotient
${X^\vee}/Q^\vee \onto {X^\vee}/\tQ$ splits, since ${X^\vee}/\tQ$ is free, and
the resulting pullback maps $\Hom({X^\vee}/Q^\vee, \bZ) \leftrightarrow
\Hom({X^\vee}/\tQ, \bZ)$ are inverse to each other since the quotient
${X^\vee}/Q^\vee \onto {X^\vee}/\tQ$ has finite kernel and $\bZ$ is torsion-free. 
(More generally, for any finite-kernel quotient of finitely-generated abelian groups, the pullback map on $\Hom(-,\bZ)$ is an isomorphism.)
\end{remark}

\subsection{Easy $W$-invariant characters}
Let $G$ be a reductive group defined over $\bZ$. 
 Let $\Hom_\sm(\cOt, \bCt)$ denote the group of smooth characters
$\cOt \to \bCt$.

\begin{proposition}  \label{p:easy}
The following conditions are equivalent
for a smooth character $\bmu: T(\cO) \ra \bCt$:
\begin{itemize}
\item[(i)] The restriction $\bmu|_{([G,G] \cap T)(\cO)}$ is trivial;
\item[(ii)] The character $\bmu$ is a product of compositions of
  $W$-invariant rational characters $T(\cO)\ra \cOt$ with smooth characters
  $\cOt\ra \bCt$.
\end{itemize}
\end{proposition}
\begin{remark}The same statement and proof holds when $\cO$ is replaced by any (topological) ring.
\end{remark}

\begin{definition} \label{d:easy} A smooth $W$-invariant character $\bmu:T(\cO)\ra
  \bCt$ is \emph{easy} (with respect to $G$) if the equivalent
  conditions of Proposition \ref{p:easy} are satisfied.
\end{definition}

\begin{remark}\label{r:easy}
By Lemma \ref{l:rationalChar} and Proposition \ref{p:easy}, the group of easy characters $T(\cO)$ identifies canonically with 
 \begin{multline*}
 \Hom_\sm( (T/(T\cap [G,G]))(\cO), \bCt)\simeq \\
  \Hom({X^\vee}/\tQ, \bZ)\otimes_\bZ \Hom_\sm(\cOt, \bCt) \simeq  \Hom_\sm({X^\vee}/\tQ \otimes_{\bZ} \cOt, \bCt). 
\end{multline*}
 The last isomorphism follows from the fact that $X^\vee/\tQ$ is free. 
 \end{remark}

\begin{proof}[Proof of Proposition \ref{p:easy}]
  The implication (ii) $\Rightarrow$ (i) is immediate from Lemma
  \ref{l:rationalChar}. For the reverse implication, note that by
  assumption $\bmu$ is a character of $T(\cO)=(\bbG_m \otimes_\bZ
  X^\vee)(\cO)$ which is trivial on $([G,G] \cap T)(\cO)=(\bbG_m
  \otimes_{\bZ} \tQ)(\cO)$. Therefore $\bmu$ is canonically a
  character of
  \begin{multline*}
  (\bbG_m \otimes_\bZ X^\vee)(\cO)/(\bbG_m \otimes_{\bZ} \tQ)(\cO) = \\
  ((\bbG_m \otimes_\bZ X^\vee)/ (\bbG_m \otimes_{\bZ} \tQ)) (\cO) = (\bbG_m\otimes_\bZ X^\vee/\tQ)(\cO) = 
  (X^\vee/\tQ)\otimes_\bZ \cOt.
  \end{multline*}
  We conclude that $\bmu$ is a product of
  compositions of (smooth) characters of $\cOt$ with rational
  characters $\Hom(X^\vee/\tQ,\bZ)$. By Remark \ref{r:Winvariant}, the group of such rational characters is canonically
  isomorphic to the sublattice $X^W$ of $W$-invariant
  rational characters. Therefore, $\bmu$ has the form claimed in Part (ii).  
  \end{proof}

  Note that, if $\bmu$ is easy, then Lemma \ref{l:rationalChar}
  implies that $\bmu$ extends to a character of $G(F)$, and hence of
  $G(\cO)$. As the following example illustrates, the converse is not,
  in general, true.

\begin{exam}\label{ex:strongpar-not-easy}
  Let $G=PGL_2$.  Then the determinant map $\GL_2(F) \to F^\times$
  descends to a map $G(F) \to F^\times / (F^\times)^2 \cong
  \{\pm 1\} \times t^{X^\vee}/t^{2X^\vee}$. Take the composition and the further quotient by the second factor, and view it as a character $G(F) \to \bCt$ (which 
is trivial on $t^{X^\vee}$).  The restriction of this character to $T(\cO)$
is nontrivial, 
 even though there are no nonzero $W$-invariant rational
  characters (and hence non nontrivial easy characters).
\end{exam}
Nonetheless, in the next subsection, we give a combinatorial
description of all characters of $T(\cO)$ which extend to characters
of $G(F)$, similar to the description of easy characters above.

\subsection{Extendable $W$-invariant characters} 

\begin{proposition}\label{p:strongpar} 
The following conditions are equivalent
for a smooth $W$-invariant character $\bmu: T(\cO) \ra \bCt$:
\begin{enumerate}
\item[(a)] $\bmu$ extends to a character of $G(F)$;
\item[(b)] For all $\alpha \in \Delta$, we have 
 \begin{equation} \label{e:bmualpha}
 \bmu\circ \alpha^\vee |_{\cOt} = 1.
 \end{equation} 
\end{enumerate}
If in addition $q > 2$, then these are also equivalent to
\begin{enumerate}
\item[(c)] $\bmu$ extends to a character of $G(\cO)$.
\end{enumerate}
\end{proposition}

\begin{definition} \label{d:extendable}
A smooth character 
$\bmu:T(\cO)\ra \bCt$ is said to be \emph{extendable} (to $G(F)$) if it satisfies the equivalent conditions of the above proposition.
\end{definition}

\begin{remark} \label{r:extendable}
  Note that $\alpha^\vee$ generate the coroot lattice $Q^\vee\subset X^\vee$; hence, $\langle \alpha^\vee(\cOt) \rangle = Q^\vee \otimes_{\bZ} \cOt$.  
  Also note that $T={X^\vee}\otimes_\bZ \bbG_m$; therefore, $T(\cO)={X^\vee}\otimes_\bZ \cOt$. 
  It follows from the above proposition that the group of extendable (to $G(F)$) characters of $T(\cO)$ identifies with 
  \begin{equation} 
 \Hom_\sm (T(\cO) / \langle \alpha^\vee(\cOt) \rangle_{\alpha \in \Delta}, \bCt) = 
\Hom_\sm(({X^\vee}\otimes_\bZ \cOt)/(Q^\vee \otimes_{\bZ} \cOt), \bCt)
 \simeq
\Hom_\sm({X^\vee}/Q^\vee \otimes_\bZ \cOt, \bCt).
\end{equation} 
The last isomorphism follows because 
$({X^\vee}\otimes_\bZ \cOt)/(Q^\vee \otimes_{\bZ} \cOt) = {X^\vee}/Q^\vee \otimes_\bZ \cOt$
(by definition of $\otimes$, or by its right-exactness).
Note that ${X^\vee}/Q^\vee \otimes_\bZ \cOt$ has a topology which is induced by the topology on $\cOt$. Therefore, we can speak of its smooth characters. 
\end{remark}

The rest of this subsection is devoted to the proof of Proposition
\ref{p:strongpar}.  To this end, we will prove the following: 
\begin{equation} \label{e:comms}
[G(\cO), G(\cO)] \cap T(\cO) = \langle \bmu \circ
  \alpha^\vee(\cOt) \rangle_{\alpha \in \Delta} \text{ if $q > 2$}, \quad
[G(F), G(F)] \cap T(F) = \langle \bmu \circ
  \alpha^\vee(F) \rangle_{\alpha \in \Delta}, \forall q.
\end{equation}
Let us explain how this implies the proposition.  First, note
that, given any character of $T(\cO)$, there is a unique extension to
a character of $T(F)$ trivial on $t^{X^\vee}$.  Moreover, if the
original character was trivial on $\alpha^\vee(\cOt)$, then the
extension is trivial on $\alpha^\vee(F^\times)$, for all $\alpha \in
\Delta$.  Applying this observation, together with the second identity
in \eqref{e:comms} and Lemma \ref{l:char-ext}, we conclude that parts
(a) and (b) of the proposition are equivalent.  The first identity in
\eqref{e:comms} similarly implies, for $q > 2$, that parts (b) and (c)
of the proposition are equivalent.

\subsubsection{Proof of \eqref{e:comms}}
\begin{lemma} Fix an arbitrary ring $R$. Then,
\begin{equation}
[G(R), G(R)] \cap T(R) \subseteq \langle \bmu \circ
  \alpha^\vee(R) \rangle_{\alpha \in \Delta}.
\end{equation}
The opposite inclusion, $[G(R), G(R)] \supseteq \langle \bmu \circ
\alpha^\vee(R) \rangle_{\alpha \in \Delta}$ holds for all $G$ if and
only if it holds when $G(\bC)$ is
semisimple and simply connected. In the latter situation, it is equivalent to
\begin{equation}\label{e:comms2}
[G(R), G(R)] \supseteq T(R).
\end{equation}
\end{lemma} 
\begin{proof} 
More generally, suppose we have an arbitrary morphism of
groups $\tG\ra G$ and a subgroup $T<G$ such that $G$ is generated by $T$ and the
image of $\tG$. Then, it is obvious that $[G,G]$ is the normal subgroup generated by 
$[\pi(\tG), \pi(\tG)] = \pi[\tG, \tG]$,  $[\pi(\tG), T]$, and $[T,T]$.

Returning to the situation of the lemma, let $\tG$ be the connected
reductive algebraic group such that $\tG(\bC)$ is the universal cover
of $[G,G](\bC)$. 
Let $\pi$ denote the canonical morphism $\tG\ra G$. Abusively, we will
use $\pi$ also to denote the induced morphism $\tG(R)\ra G(R)$. Note
that $\pi$ is an isomorphism on root subgroups; therefore, $G(R)$ is
generated by $\pi(\tG(R))$ and $T(R)$. Applying the considerations of
the previous paragraph, we conclude that $[G(R),G(R)]$ is the normal
subgroup generated by the image $\pi([\tG(R), \tG(R)])$, along with
$[T(R), \pi(\tG(R))]$ and $[T(R), T(R)]$. Note that these are all contained
in the normal subgroup generated by $\pi(\tG(R))$.

We claim, however, that $\pi(\tG(R))$ is normal, so that $[G(R), G(R)]
\subseteq \pi(\tG(R))$.  Since $G(R)$ is generated by $\pi(\tG(R))$ and
$T(R)$, it suffices to show that $\pi(\tG(R))$ is closed under
conjugation by $T(R)$. This follows because $\pi(\tG(R))$ is generated
by the root subgroups and $\pi(\tT(R))$, and each of these are closed
under conjugation by $T(R)$. Therefore, for $\tilde T < \tilde G$ the
maximal torus of $\tilde G$, we conclude that
\begin{equation} 
[G(R), G(R)] \cap T(R) \subseteq \pi (\tilde T(R)) = \langle \bmu \circ
  \alpha^\vee(R) \rangle_{\alpha \in \Delta}.
\end{equation}
The final equality holds because, for a simply connected semisimple
root system, the maximal torus is generated by the coroot subgroups.

The second assertion follows since $[G(R), G(R)] \supseteq \pi
[\tG(R), \tG(R)]$ and the set $\Delta$ of roots is the same for $G$
and $\tG$.  For the final assertion (the case $G = \tG$), we only need
to note that, in this case, $T$ is generated by the coroot subgroups,
so $T(R) = \langle \bmu \circ \alpha^\vee(R) \rangle_{\alpha \in
  \Delta}$.  
\end{proof}

\begin{lemma} For a fixed ring $R$, the inclusion \eqref{e:comms2}
  holds for all $G$ such that $G(\bC)$ is semisimple and
  simply connected if and only if it holds for $G = \SL_2$.
\end{lemma}
\begin{proof} Suppose $G(\bC)$ is semisimple and simply
  connected. Then $T(\cO)$ is generated by its coroot subgroups, so it
  is enough to show \eqref{e:comms2} for the image of all subgroups
  $\SL_2 \to G$ corresponding to roots of $G$.
\end{proof}

\begin{lemma} \label{l:sl2-tech}
The inclusion \eqref{e:comms2} holds for $G = \SL_2$ in the case that either $R = \cO$ for $q > 2$ or $R$ is a field.
\end{lemma} 
\begin{proof} 
Let $f \in R^\times$ and let $\alpha$ be the positive
simple root. Then
one can verify that
\begin{equation} \label{e:expl-comm}
\begin{pmatrix} 1 & 0 \\ -(f-1)^2/f & 1 \end{pmatrix}
\Bigl[ \begin{pmatrix} 1 & 0 \\ f-1 & 1 \end{pmatrix},
\begin{pmatrix} 1 & -1 \\ 0 & 1 \end{pmatrix} \Bigr] 
\begin{pmatrix} 1 & (1-f)/f \\ 0 & 1 \end{pmatrix}
= \begin{pmatrix} f & 0 \\ 0 & f^{-1} \end{pmatrix} = \alpha^\vee(f). 
\end{equation} 
We now consider the question of when the first and last matrices on
the LHS are in $[G(R), G(R)]$. Generally, for $g \in R^\times$,
\begin{equation} \label{e:comms2-tech}
\Bigl[
\begin{pmatrix}
1 & 0 \\
x & 1
\end{pmatrix},
\begin{pmatrix}
g & 0 \\
0 & g^{-1}
\end{pmatrix}
\Bigr]
=
\begin{pmatrix}
1 & 0 \\
x (1-g^2) & 1
\end{pmatrix}.
\end{equation}
Let $I_R := (1-g^2 \mid g \in R) = \{x(1-g^2) \mid x,g \in R\}$ be the ideal
of elements appearing in the lower-left entry of the final matrix. If this
ideal is the unit ideal, then the LHS of \eqref{e:expl-comm} is in $[G(R),
G(R)]$, as desired.  This is clearly true if $R$ is a field such that $|R| > 3$.

We claim that $I_R$ is also the unit ideal
 when $R = \cO$ and $q > 3$.  First, more generally for $q > 2$, we claim that $I_\cO \supseteq \fp$.  Indeed, the squaring operation is bijective on $1+\fp$, for all $z \in \fp$. So, we can take $g \in \cO$
such that $g^2 = 1 +z$, and hence $z \in I_R$.  

So for $q > 3$, to show $I_\cO$ is the unit ideal reduces to showing
that $I_{\cO/\fp} = I_{\Fq}$ is the unit ideal, which as we pointed
out is true in this case.

It remains to show that, for $R = \cO$ for $q=3$, or $R=\Fq$ for $q
\leq 3$, that \eqref{e:comms2} holds.  We have to show this slightly
differently, since now $I_R$ is not the unit ideal.

First, if $R = \Fq$ and $q=2$, there is nothing to show because now
$T(\Fq)$ is trivial. If $R = \Fq$ for $q=3$, then it is well known
that $[G(\Fq), G(\Fq)]$ has index three in $G(\Fq)$; since $T(\Fq)$
has order two, it follows that $T(\Fq)$ must be in the kernel of the
abelianization map $G(\Fq) \to G(\Fq) / [G(\Fq), G(\Fq)]$, i.e., that
$[G(\Fq), G(\Fq)] \supseteq T(\Fq)$.  This completes the proof of the
lemma for $R$ equal to a field.

Next, suppose that $R=\cO$ and $q=3$.  Then,
it suffices to show that $[G(\cO), G(\cO)] \supseteq \alpha^\vee(1+\fp)$.  Indeed, if we show this, the inclusion reduces to $[G(\Fq), G(\Fq)] \supseteq \alpha^\vee(\Fq^\times)$, which we have now established.

To show this, note that, by the above argument, $I_\cO \supseteq \fp$.
Hence, we can apply \eqref{e:comms2-tech} to the case $f \in 1 + \fp$,
and we conclude that $\alpha^\vee(1+\fp) \supseteq [G(\cO), G(\cO)]$,
as desired. \qedhere
\end{proof}

\subsection{Comparison between easy and extendable}\label{ss:comparison}  
\begin{corollary} \label{c:implications}
 Let $\bmu:T(\cO)\ra \bCt$ be a smooth character of $G$. Then 
\[
  \bmu \textrm{ is easy for $G$} \implies
  \bmu \textrm{ is extendable to $G(F)$ }
  \implies
\bmu \textrm{ is $W$-invariant} \implies (\bmu\circ \alpha^\vee|_{\cOt})^2=1, \, \forall
  \alpha\in \Delta.
\]
\end{corollary} 
\begin{proof} 
  The first implication is immediate from Propositions \ref{p:easy} and \ref{p:strongpar}.  The second
  implication follows from the facts $W\simeq
  N_{G(\cO)}(T(\cO))/T(\cO)$ and $[N_{G(\cO)}(T(\cO)),T(\cO)]\subseteq
  [G(\cO),G(\cO)]\cap T(\cO)$. For the last implication, note that
  $(\bmu\circ \alpha^\vee)^2(x) = \bmu([\alpha^\vee(x), s_\alpha])$,
  where $s_\alpha$ is any lift to $N_{G(\cO)}(T(\cO))$ of the simple
  reflection $s_\alpha$.
  \end{proof} 
The reverse implications can all fail.  For the first implication, see
Example \ref{ex:strongpar-not-easy}.
  For the remaining two, we have the following:

\begin{exam} \label{ex:stronglyPar}
\begin{enumerate} 
\item[(i)] Let $G=\SL_2$. Let $\bmu:T(\cO)\ra \bCt$ denote the composition 
\[
T(\cO)\simeq \cOt \ra \cOt/(\cOt)^2 \rar{\theta} \bCt,
\]
where $\theta$ is a nontrivial character.
Then $\bmu:T(\cO)=\cOt\ra \bCt$ is $W$-invariant; however, it does not
extend to $G(F)$ by Proposition \ref{p:strongpar}.
\item[(ii)]\cite[Example 8.4]{Roche98} Let $G=\Sp_{2n}$, $n\geq
  2$. Identify $T(\cO)$ with $(\cOt)^n$, and let $\bmu=(\theta,
  \cdots, \theta)$. Then $\bmu$ is $W$-invariant; however, it does not
  extend to $G(F)$. This is because, as observed in \cite[Example
  8.4]{Roche98}, the composition $\bmu \circ \alpha^\vee$ is not
  trivial for all $\alpha$ (and in fact, the root subsystem whose
  coroots have trivial composition produces an
  endoscopic group $\mathrm{SO}_{2n}$, which is not a subgroup of $G$).
\end{enumerate} 
\end{exam}

\begin{exam}\label{ex:sl3-nonpar}
Let $G=\SL_3$.  Define 
\[
\bmu(\diag(a,b,a^{-1}b^{-1}))=\theta(a)\theta(b),\quad a,b\in \cOt,
\] where $\theta$ is a nontrivial quadratic character of $\cOt$.
 By assumption, $(\bmu\circ \alpha^\vee)^2=1$ for both coroots of
$G$; however, $\bmu$ is not invariant under the transformation
$(a,b,a^{-1}b^{-1})\mapsto (a^{-1}b^{-1} ,b,a)$; in particular, it is
not $W$-invariant. 
\end{exam} 

In certain situations, either (or both) of the first two implications
in the above corollary become biconditionals.  
\begin{lemma} \label{l:cc-sc-strongpar-easy} 
\begin{itemize}
\item[(i)] Suppose that $Q^\vee = \langle \lambda - w(\lambda) \mid \lambda \in {X^\vee, w \in W} \rangle$.   Then  every $W$-invariant character of $T(\cO)$ is extendable to $G(F)$.

\item[(i')] The hypothesis of (i) is equivalent to the statement that,
 for some choice of simple
roots $\alpha_i$, there exist cocharacters $\lambda_i \in X^\vee$ such that $\langle \lambda_i, \alpha_i \rangle = 1$.   Moreover, this condition is implied by either of
the following:
\begin{itemize}
\item[(a)] $X/Q$ is free
\item[(b)] The root system of $G$ has no factors of type $A_1$ or $C_n$.
\end{itemize}

\item[(ii)] Suppose $X^\vee/Q^\vee$ is torsion-free. Then every extendable character of $T(\cO)$ (to $G(F)$) is easy.
\end{itemize}
\end{lemma}

\begin{proof}
  (i) If the coroot lattice equals the span of the elements $\lambda -
  w(\lambda)$ for $w \in W$ and $\lambda \in {X^\vee}$, then \eqref{e:bmualpha} is satisfied. This is because $W$-invariance implies $\bmu(\lambda(x)) = \bmu(w(\lambda)(x))$ for
all $x \in \Gm(\cO)$, and hence $\bmu((\lambda-w(\lambda))(x)) = 1$ for all $x \in \Gm(\cO)$.

(i') First, we claim that
$Q^\vee \supseteq \langle \lambda - w(\lambda) \mid \lambda \in X^\vee, w \in W\rangle$.  Let $\alpha_i, i \in I$ be a choice of simple roots. Since  $W$ is generated by the $s_{\alpha_i}$, 
\[
\langle \lambda - w(\lambda) \mid \lambda \in X^\vee, w \in W \rangle = \langle \lambda - s_{\alpha_i} \lambda \mid \lambda \in X^\vee, i \in I \rangle = \langle \langle \lambda, \alpha_i \rangle \alpha_i^\vee \mid \lambda \in X^\vee, i \in I \rangle.
\]
This proves the desired containment.   So, we need to show that the opposite
inclusion is equivalent to the condition stated in (i').  

Given $\lambda_i$ such that $\langle \lambda_i, \alpha_i \rangle = 1$,
we obviously get $\alpha_i^\vee$ in the RHS 
of the above equation.  Conversely, if $\alpha_i^\vee
\in \langle \langle \lambda, \alpha_i \rangle \alpha_i^\vee \mid \lambda
\in X^\vee, i \in I \rangle$, then there must exist $\lambda_i \in
X^\vee$ such that $\langle \lambda_i, \alpha_i \rangle = 1$.  Applying
this to all $i$ yields the desired equivalence (since $Q^\vee$ is
spanned by the $\alpha_i^\vee$).

(a) If $X/Q$ is torsion-free, then $Q$ must be saturated in $X$, 
so the condition (i') is satisfied.

  (b) For the root system $A_2$, with simple roots $\alpha_1$ and
  $\alpha_2$, one has $s_{\alpha_1}(\alpha_2^\vee) -\alpha_2^\vee =
  \alpha_1^\vee$, and similarly with indices $1$ and $2$ swapped, so
  that one concludes that $\alpha_1^\vee, \alpha_2^\vee \in \langle
  \lambda - w(\lambda) \rangle$ and hence $Q^\vee = \langle \lambda -
  w(\lambda) \rangle$.  The same argument shows that, for every root
  system in which every simple root is contained in a root subsystem
  of type $A_2$, then every coroot is contained in $\langle \lambda -
  w(\lambda) \rangle$ and hence (i) is also satisfied.  

  This takes care of all root systems except for types $A_1, B_n,
  C_n$, and $G_2$.  For type $B_n$ with $n \geq 3$, the above argument
  shows that, for the standard choice of simple roots $\alpha_1,
  \ldots, \alpha_n$ where $\alpha_n$ is the short simple root, then
  $\alpha_i^\vee \in \langle \lambda - w(\lambda) \rangle$ for $i <
  n$, since these are incident to a subdiagram of type $A_2$; for
  $\alpha_n^\vee$, it is still true that
  $s_{\alpha_n}(\alpha_{n-1}^\vee)-\alpha_{n-1}^\vee = \alpha_n^\vee$,
  so also $\alpha_n^\vee \in \langle \lambda - w(\lambda) \rangle$.
  For type $G_2$, if the simple roots are $\alpha_1$ and $\alpha_2$,
  we see that $s_{\alpha_1}(\alpha_1^\vee+\alpha_2^\vee) -
  (\alpha_1^\vee+\alpha_2^\vee) = \pm\alpha_1^\vee$, so $\alpha_1^\vee
  \in \langle \lambda - w(\lambda) \rangle$, and the same fact holds
  (with opposite sign) when indices $1$ and $2$ are swapped.  Note
  also that $B_2=C_2$, so we do not need to separately exclude $B_2$.

  (ii)  The hypothesis is equivalent to the condition that $Q^\vee$ is saturated in $X^\vee$; i.e., $Q^\vee=(Q^\vee)_\sat$. The result then follows from Remarks \ref{r:easy} and \ref{r:extendable}. 
\end{proof}

\subsection{On parabolic characters} Recall that a smooth character $\bmu:T(\cO)\ra \bCt$ is parabolic if its stabilizer in $W$ is a parabolic subgroup. Here is an example of a character which is not parabolic. 

\begin{exam}(cf.~\cite[Example 8.3]{Roche98}, due to Sanje-Mpacko) \label{ex:parabolic}
Let $N\geq 3$ and $G=\SL_N$. Define
\[
\bmu(\diag(a_1,a_2,\ldots,a_{N-1},a_1^{-1}\cdots a_{N-1}^{-1} )) =
\chi(a_1)\chi^2(a_2) \cdots \chi^{N-1}(a_{N-1}),
\]
where $\chi: \cOt \to \bCt$ is a character of order $N$.  Then the
stabilizer of $\bmu$ in $W$ is the subgroup $\bZ/N$ of cyclic
permutations, which is not parabolic (and in particular is not all of
$W$).
\end{exam}

On the other hand, as the following proposition illustrates, in certain situations all characters are parabolic.

\begin{proposition} \label{p:parabolicADE} Let $G$ be a connected simply laced split reductive group. If $X/Q$ is free, then every smooth character of $T(\cO)$ is strongly parabolic.  If, moreover, $X^\vee/Q^\vee$ is free, then every smooth character of $T(\cO)$ is easy. 
\end{proposition}

\begin{proof} Let $\Delta_\bmu$ denote the collection of roots
  $\alpha\in \Delta$ such that $\bmu\circ \alpha^\vee =1$. We claim
  that $\Delta_\bmu$ is a closed root subsystem.  Indeed, if $\alpha,
  \beta \in \Delta_\bmu$ and $\langle \alpha, \beta \rangle = -1$,
  then $(\alpha+\beta)^\vee = \alpha^\vee + \beta^\vee$, and so
  $\alpha+\beta \in \Delta_\bmu$ as well. Let $L$ denote the Levi
  subgroup corresponding to $\Delta_\bmu$. It follows from Proposition
  \ref{p:strongpar} along with \cite{Roche98}, Lemma 8.1.(i) and the
  comment at the end of p. 395, that $\bmu$ is strongly parabolic with
  Levi $L$ (cf.~Remark \ref{r:roche-th}).  Alternatively, if we use
  only from \cite{Roche98} that $\bmu$ is parabolic, then we can apply
  Lemma \ref{l:cc-sc-strongpar-easy} to deduce strong parabolicity.
  For the final statement, we again apply Lemma
  \ref{l:cc-sc-strongpar-easy}. \end{proof}

\subsection{Proof of Theorem \ref{t:characters}}
Parts (i) and (ii) follow from Propositions \ref{p:easy} and
\ref{p:strongpar}, respectively.  Next, we need a basic fact from the
theory of reductive groups.

\begin{lemma} \label{l:torsion} Let $G$ be a connected split reductive group over $\bZ$ with split torus $T$. Let $L<G$ be a Levi containing $T$. 
\begin{enumerate} 
\item[(i)] If ${X^\vee}/Q^\vee$ is torsion free, so is ${X^\vee}/Q_L^\vee$. 
\item[(ii)] If ${X}/Q$ is torsion free, then so is ${X}/Q_L$.
\item[(iii)] If the equivalent conditions (i) or (i') of Lemma
  \ref{l:cc-sc-strongpar-easy} are satisfied for $G$ (i.e., $Q^\vee =
  \langle \lambda - w(\lambda) \mid \lambda \in {X^\vee} \rangle$ or,
  for some choice of simple roots $\alpha_i$, there exist 
  cocharacters $\lambda_i \in X^\vee$ such that $\langle \lambda_i, \alpha_i
  \rangle = 1$), then they are also satisfied when $G$ is replaced by
  $L$.
\end{enumerate}   
  \end{lemma} 
  \begin{proof} Parts (i) and (ii) follow from the fact that $Q/Q_L$
    and $Q^\vee/Q^\vee_L$ are always torsion free.  For part (iii), we
    consider the condition (i'), i.e., the second condition. Note
    that, if this condition is satisfied for some choice of simple
    roots, it must be satisfied for all choices of simple roots, since
    two different choices are related by an element of the Weyl group.
    But, by definition, one can choose simple roots of $L$ which form
    a subset of a choice of simple roots of $G$ (and note that the (co)weight lattices are the same for $L$ as for $G$).
  Hence condition (i') is satisfied for $L$.
  \end{proof}

  Then, parts (iii) and (iv) both follow from Lemmas \ref{l:torsion} and
  \ref{l:cc-sc-strongpar-easy}. Finally, part (v) follows from Proposition
  \ref{p:parabolicADE} and Lemma \ref{l:torsion}.


\section{Central families and Satake Isomorphisms} \label{s:Satake}

\subsection{Recollections on decomposed subgroups} We begin this section with some general remarks on compact open subgroups of $G(F)$.  
Let $P$ be a parabolic subgroup of $G$ with Levi decomposition
$LU_P^+$. Let $P^-=L U_P^-$ denote the opposite of $P$ relative to
$L$. (According to \cite[Proposition 14.21]{Borel}, the opposite
parabolic is unique up to conjugation by a unique element of $U_P^+$.)
Let $J\subset G$ be a compact open subgroup. Let
  \[
  J_P^+ = J\cap U_P^+(F), \quad J_P^0 = J\cap L(F), \quad J_P^-=J\cap
  U_P^-(F).
  \]

For a parabolic $P=LU_P^+$,
we let $\Delta_P^+$ denote the set of roots of $U_P^+$. Similarly,
we let $\Delta_P^-$ denote the set of roots of $U_P^-$. Note that
$\Delta = \Delta_L \sqcup \Delta_P^+ \sqcup \Delta_P^-$.
\begin{definition} \label{d:decomposed}
\begin{enumerate} 
\item The subgroup $J$ is \emph{decomposed} with respect to $P$ if the
  product
\[
J_P^+\times J_P^0 \times J_P^- \ra J
\]
is surjective (and hence bijective). 
\item The group $J$ is \emph{totally decomposed} with respect to $P$
 if it is decomposed, and in addition, the product maps
\[
\prod_{\alpha \in \Delta_P^\pm} U_\alpha(F)\ra J_P^{\pm}
\]
are surjective (and hence bijective) for any ordering of the factors on
the left hand side.
\item We say that $J$ is \emph{absolutely totally decomposed} if it
is totally decomposed with respect to all parabolic subgroups $P$.
\end{enumerate} 
\end{definition}

The above definitions are closely related to the ones given in
\cite[\S 6]{Bushnell98} and \cite[\S 1.1]{Bushnell00}. (Note, however,
that similar decompositions appear in \cite[\S
6]{BTgrcl1}.)  The following result, which is immediate from the
definitions, is similar to a statement in \cite[\S 1.1]{Bushnell00}.

\begin{lemma} \label{l:decomposed1} Let $J$ be totally decomposed in
  $G$ with respect to a Borel subgroup $B$. Then $J$ is totally
  decomposed with respect to every parabolic $P$ containing $B$.
\end{lemma}

In particular, if $J$ is totally decomposed with respect to all
Borels, then it is absolutely totally decomposed. The following lemma is also immediate from the definitions. 

\begin{lemma} \label{l:decomposed2} Let $J$ be a compact open subgroup
  of $G$ which is decomposed with respect to a parabolic $P$. Suppose
  $L(\cO)$ normalizes $J_P^+$ and $J_P^-$. Then the subset $K=JL(\cO)$
  is a subgroup of $G(F)$; moreover, it is decomposed with respect to
  $P$; that is, $K=K_P^+ K_P^0 K_P^-$ where $K_P^\pm=J_P^\pm$ and
  $K_P^0=L(\cO)$.
\end{lemma}

\subsection{The subgroup K}\label{s:K}
Let $f: \Delta \to \bZ$ be a function satisfying the
properties
\begin{enumerate}\label{eq:rocheGroups}
\item[(a)] $f(\alpha) + f(\beta) \geq f(\alpha+\beta)$, whenever
  $\alpha, \beta, \alpha+\beta \in \Delta$;
\item[(b)] $f(\alpha) + f(-\alpha) \geq 1$.
\end{enumerate}
In particular, $f$ is concave in the sense of Bruhat and Tits (see
\cite{BTgrcl1}, \S 6.4.3 and \S 6.4.5).  Let
\begin{equation}\label{eq:Jf}
  J=J_f := \langle U_{\alpha, f(\alpha)}, T(\cO) \mid \alpha \in \Delta \rangle.
\end{equation}
Using the results of Bruhat and Tits, specifically \cite[Proposition
6.4.9]{BTgrcl1}, Roche proved the following lemma.

\begin{lemma}\cite[Lemma 3.2]{Roche98} \label{l:uroche} 
  The group $J$ is absolutely totally decomposed in $G$. Moreover, $J \cap U_\alpha(F) = U_{\alpha, f(\alpha)}$ for all
  $\alpha \in \Delta$.
\end{lemma}

Next, let $L$ be a Levi subgroup of $G$ (by which we always mean a Levi
for a parabolic subgroup containing $T$). We are interested to know when
$K=JL(\cO)$ is a group. In view of Lemma \ref{l:normalize}, it is
enough to check that $L(\cO)$ normalizes $J_P^\pm$ for some choice of
parabolic $P$ with Levi component $L$. 

\begin{lemma} \label{l:normalize} Let $P$ be a parabolic with Levi
  component $L$.  Suppose that
\begin{equation}\label{e:f-condc}
f(\beta) = f(\alpha+\beta), \quad \forall \alpha \in \Delta_L,  \beta \in \Delta \setminus \Delta_L \text{ such that } \alpha+\beta \in \Delta.
\end{equation}
Then $L(\cO)$ normalizes $J_P^\pm$.
\end{lemma} 
To prove the above, we will make use of the following lemma, which will
also be useful later:
\begin{lemma}\label{l:fprime}
Assume that $g: \Delta\setminus \Delta_L \to \bZ$ satisfies \eqref{e:f-condc}, in addition
to conditions (a) and (b) (restricting $\alpha, \beta$, and $\alpha+\beta$ to lie in $\Delta \setminus \Delta_L$).
Let $\Delta_L^+ \subseteq \Delta_L$ be any choice of positive roots, and
let $f: \Delta \to \bZ$ be the function defined by
\begin{equation}
f|_{\Delta \setminus \Delta_L} = g, \quad
f|_{\Delta_L^+} = 0, \quad f|_{\Delta_L^-} = 1.
\end{equation}
Then $f$ satisfies conditions (a) and (b). 
\end{lemma}
Note that $J_{f|_{\Delta_L}} = I_L$ is the Iwahori subgroup of $L(\cO)$
corresponding to $\Delta_L^+ \subseteq \Delta_L$.
\begin{proof}
  It is clear (and standard) that $f|_{\Delta_L}$ satisfies
  conditions (a) and (b) (where we require in (a) that $\alpha,
  \beta$, and $\alpha+\beta$ lie in $\Delta_L$). By hypothesis,
  $f|_{\Delta \setminus \Delta_L}$ satisfies conditions (a) and (b)
  (requiring $\alpha, \beta$, and $\alpha+\beta$ to be in
  $\Delta \setminus \Delta_L$ in (a)).  So we only need to check that,
  if $\alpha \in \Delta_L$ and $\beta \in \Delta \setminus \Delta_L$,
  then condition (a) is satisfied in the case that $\alpha+\beta \in
  \Delta$.  This is immediate from \eqref{e:f-condc}.
\end{proof}
\begin{proof}[Proof of Lemma \ref{l:normalize}]
  Choose a subset $\Delta_L^+ \subseteq \Delta_L$ of positive roots
  for $L$.  Let $g = f|_{\Delta \setminus \Delta_L}$, and let $f':
  \Delta \to \bZ$ be as in Lemma \ref{l:fprime} (i.e., $f'|_{\Delta
    \setminus \Delta_L} = f|_{\Delta \setminus \Delta_L}$,
  $f'|_{\Delta_L^+} = 0$ and $f'|_{\Delta_L^-}=1$).  Let $I_L =
  J|_{f'|_{\Delta_L}} < L(\cO)$ be the corresponding Iwahori subgroup
  containing $T(\cO)$.  Then $I_L \leq J_{f'}$, and hence $I_L$
  normalizes $J_{f'}$.  It also normalizes $U_P^\pm$ (since $L$ normalizes the unipotent radical $U_P^+$), so $I_L$
  normalizes $J_{f'} \cap U_P^{\pm} = J_f \cap U_P^{\pm} = J_P^\pm$.
  On the other hand, $L(\cO)$ is generated by all its Iwahori
  subgroups, so $L(\cO)$ also normalizes $J_P^\pm$.  (Note that we
  could have also used the decomposition $L(\cO)=I_L W_L I_L$, for
  $W_L$ the Weyl group of $L$, and the fact that $W_L$ normalizes
  $J_P^+$ under hypothesis \eqref{e:f-condc}.) \qedhere
\end{proof}

\begin{proposition} \label{p:IwahoriK} Let $L$ be a Levi
  subgroup of $G$. Assume that the function $f:\Delta\ra \bZ$
  satisfies conditions (a) and (b) as well as
\eqref{e:f-condc}, and set $J = J_f$.
 Then $K=JL(\cO)$
is a group; moreover, $K$ is decomposed with respect to every
parabolic $P$ with Levi $L$.
\end{proposition} 

\begin{proof} By Lemma \ref{l:normalize}, $L(\cO)$ normalizes $J_P^+$
  and $J_P^-$.  The result follows then from Lemma \ref{l:decomposed2}.
\end{proof}

The following corollary gives an alternative definition of $K$.

\begin{corollary} Let $L$ be a Levi subgroup of $G$. Let $g:\Delta\ra \bZ$ be a function satisfying the following properties:
\begin{enumerate} 
\item[(i)] $g(\alpha)=0$ for all $\alpha \in \Delta_L$;
\item[(ii)] $g(\alpha) + g(-\alpha) \geq 1$ for all $\alpha\in \Delta\setminus \Delta_L$;
\item[(iii)] $g(\alpha) +g(\beta) \geq g(\alpha+\beta)$, whenever
  $\alpha, \beta, \alpha+\beta \in \Delta$.
\end{enumerate} 
Then $K=\langle U_{\alpha, g(\alpha)}, T(\cO) \rangle$ is a compact
open subgroup of $G$. Moreover, $K\cap U_{\alpha}= U_{\alpha,
  g(\alpha)}$. Finally, $K=L(\cO) J_f$, where $f:\Delta\ra \bCt$ is defined
  from $g|_{\Delta \setminus \Delta_L}$ by Lemma \ref{l:fprime} (for any choice of positive roots
 $\Delta_L^+ \subseteq \Delta_L$).
\end{corollary} 

\begin{proof} The inclusion $K\subseteq J_{f}L(\cO)$ is clear. For the
  reverse inclusion, note that it is obvious that $J_{f}^\pm \subset K$,
  so we only need to show that $L(\cO)\subset K$. This follows from
  the fact that $L$ is generated by $T$ and the root subgroups. Thus,
  $K=J_{f} L(\cO)$. In particular, by the above proposition, we have a
  direct product decomposition $K=K_P^+K_P^0 K_P^-$ for every
  parabolic with Levi $L$. This implies that for $\alpha \in
  \Delta_P^\pm$, $K\cap U_\alpha=U_{\alpha, g(\alpha)}$. On the other
  hand it is clear that for $\alpha \in \Delta_L$, we have $K\cap
  U_{\alpha} = U_{\alpha, 0}$ since $U_{\alpha,0}\subset L(\cO)$.
\end{proof}

\subsection{Extension of $\bmu$} 
Let $\bmu:T(\cO)\ra \bCt$ be a smooth character. Following Roche
\cite{Roche98}, we define a compact open subgroup $J$ associated to
$\bmu$. To this end, we have to choose a partition $\Delta=\Delta^+
\sqcup \Delta^-$. (Note that this amounts to choosing a Borel
$B\subset G$.) For every $\alpha \in \Delta$, let
\begin{equation}
c_\alpha := \cond(\bmu \circ \alpha^\vee)
\end{equation}
denote the conductor of $\bmu \circ \alpha^\vee$; that is, the
smallest positive integer $c$ for which
$\bmu(\alpha^\vee(1+\fp^c))=\{1\}$.  Let
\begin{equation}\label{eq:fbmu}
  f_{\bmu}(\alpha) = \begin{cases}
    [ c_\alpha/2 ],
    & \text{if $\alpha > 0$,} \\
    [ (c_\alpha+1)/2], & \text{if $\alpha < 0$.}
    \end{cases}
\end{equation}
\begin{lemma}\cite[\S 3]{Roche98}
  Suppose that characteristic of $\Fq$ satisfies the conditions in
  \eqref{eq:char}.  Then $f_{\bmu}$ satisfies conditions (a) and (b)
  of \eqref{eq:rocheGroups}.
\end{lemma} 

In particular, in view of Lemma \ref{l:uroche}, we obtain an
associated compact open subgroup $J=J_\bmu=J_{f_{\bmu}}$. Note that
the function $f_\bmu$ and the corresponding group $J_\bmu$ depend on
the partition of $\Delta$ into positive and negative roots (or
equivalently, on the chosen Borel $B$). While we ignore this in the
notation, the reader should keep in mind that the Borel $B$ is
present. In particular, we have a decomposition $J=J^+ J^0 J^-$, where
$J^\pm=J_B^\pm$.  Let $J^\bullet = \langle J^+, J^- \rangle$. 

\begin{lemma} \label{l:muJ}\cite[\S 3]{Roche98} There exists a unique character $\mu^J:
  J\ra \bCt$ whose restriction to $J^0=T(\cO)$ equals $\bmu$ and whose
  restriction to $J^\bullet$ is trivial.
\end{lemma}

Let $\bmu$ be a strongly parabolic character of $T(\cO)$ and let $L$
denote the corresponding Levi. Let $P$ be a parabolic for $L$, and $B$
the Borel subgroup of $P$.
In terms of $B$, let
$f=f_\bmu$ denote the function associated by Roche, and let $J=J_\bmu$
denote the corresponding compact open subgroup of $G(F)$.

\begin{lemma} \label{l:K} The set $K=JL(\cO)$ is a compact open
  subgroup of $G(F)$. Moreover, for every parabolic subgroup $P$
  containing $L$, we have a decomposition $K=K_P^+ K_P^0 K_P^-$ where
  $K_P^\pm = J_P^\pm$ and $K_P^0=L(\cO)$.
\end{lemma} 

\begin{proof} 
If $\alpha\in \Delta_L$,
  then $\bmu\circ \alpha^\vee$ is trivial by \eqref{e:bmualpha}.
Now, if $\beta \in \Delta$ is such that $\alpha+\beta \in \Delta$, then
$(\alpha+\beta)^\vee = a \alpha^\vee + b \beta^\vee$ where $a$ and $b$ are
relatively prime to $q$ (by our assumption on the characteristic of $\Fq$; see Conventions \ref{n:char}).  
Therefore, for every
  $\beta\in \Delta$ such that $\alpha+\beta\in \Delta$, the
  conductor of $\bmu\circ (\alpha+\beta)^\vee$ equals the
  conductor of $\bmu\circ (\beta^\vee)$; i.e.,
\begin{equation}\label{eq:fbmu2}
f_\bmu(\alpha+\beta)=f_\bmu(\beta).
\end{equation} 
The result then follows from  Proposition \ref{p:IwahoriK}.
\end{proof} 

Let $B_L=B \cap L$ denote the corresponding Borel subgroup of $L$. Let
$I_L$ denote the corresponding Iwahori subgroup of $L$. Note that by
\eqref{eq:fbmu}, we have $J_P^0=J\cap L(\cO) = I_L$. Let $\mu^{L(F)}$
denote an extension of $\mu$ to $L(F)$. Let $\mu^{L} = \mu^{L(\cO)} :=
\mu^{L(F)}|_{L(\cO)}$ denote its restriction to $L(\cO)$. Note that
$\mu^L$ is automatically trivial on $I_L^+$ and $I_L^-$, since these
groups are in $[L(F), L(F)]$. Set $K^\bullet_P = \langle K^+_P,
K_P^-\rangle$.

\begin{proposition} \label{p:mu} There exists a unique character
  $\mu=\mu^K: K\ra \bCt$ whose restrictions to $K_P^\bullet$, $J$ and
  $L(\cO)$ equal $1$, $\mu^J$, and $\mu^L$, respectively.
\end{proposition}

\begin{proof} We need the following elementary fact: let $H^+, H^0,
  H^-$ be subgroups of a group $H$ which generate the group. 
Suppose that $H^0$ normalizes
  $H^\pm$. Let $\chi$ be a character of $H^0$ which is trivial on
  $\langle H^+, H^-\rangle\cap H^0$. Then the map $\tchi: H\ra \bCt$
  defined by $\tchi(h^+h^0h^-)=\chi(h^0)$ is a well-defined extension
  of $\chi$ to $H$.

  By assumption the characters $\mu^J$ and $\mu^L$ agree on $J\cap
  L(\cO)=I_L$; in particular, $\mu^L$ is trivial on $K_P^\bullet \cap
  L(\cO)$ (since $\mu$ is trivial on $J^\bullet$). Applying the above
  fact, we conclude that there exists a character $\mu:K\ra \bCt$
  whose restriction to $K_P^\pm$ is trivial and whose restriction to
  $L(\cO)$ equals $\mu^L$. The latter statement implies that the
  restriction of $\mu$ to $I_L^\pm$ is trivial; hence, the restriction
  of $\mu$ to $J^\pm=K_P^\pm I_L^\pm$ is also trivial. Moreover, the
  restriction of $\mu$ to $T(\cO)$ equals $\bmu$. By Lemma
  \ref{l:muJ}, the restriction of $\mu$ to $J$ equals $\mu^J$.
\end{proof}

\subsection{Proof of Theorem \ref{t:main2}} \label{ss:main2} Let
$\bmu:T(\cO)\ra \bCt$ be a strongly parabolic character of $T(\cO)$
with Levi $L$, and extensions $\mu^{L(F)}$ and $\mu^{L(\cO)} = \mu^L$
as above. Pick a parabolic $P$ containing $L$ and a Borel $B < P$.
Let $J=J_\bmu$ denote the compact open subgroup associated by Roche to
$B$ and $\bmu$. Let $\mu^J:J\ra \bCt$ denote the canonical extension
of $\bmu$ to $J$. Let
\begin{equation} 
\sW:=\ind_J^{G(F)} \mu^J. 
\end{equation} 
By definition, $\sW$ is realized on the space of left
$(J,\mu^J)$-invariant compactly supported functions on $G(F)$. The
group $G(F)$ acts on this space by right translation.  Let $f_0$ be
the function supported on $J$ which there equals $\mu^J$.
Then $f_0\in \sW$; moreover, every element of $\sW$ can be
written as a finite linear combination of elements of the form
$f_0.g$, $g\in G(F)$.

Note that $J\cap L(\cO)=I_L$ is the Iwahori subgroup of $J$. Let $\mu^I$
denote the restriction of $\mu^J$ to $I_L$.  Let $P^\circ_I:=I_L
U_P^+(\cO)$. The character $\mu^I$ extends uniquely to a character of
$P^\circ_I$ which is trivial on $U_P^+(\cO)$. By an abuse of notation,
we denote this character of $P^\circ_I$ by $\mu^I$ as well. Let
\begin{equation} 
\Pi:= \ind_{P^\circ_I}^{G(F)} \mu^I.\footnote{It is easy to check that $\Pi$, thus defined, is isomorphic to the $\Pi$ defined in \eqref{eq:IntroPhi}.}
\end{equation} 
Then $\Pi$ is realized as the space of left $(P^\circ_I,
\mu^I)$-invariant functions on $G(F)$. The group $G(F)$ acts by right
translation.  Define $a_0 : G(F) \to \bC$ to be the function supported
on $P^\circ_I J$ such that
\[
a_0(pj) = \mu^I(p) \mu^J(j), \quad p \in P^\circ_I, j \in J.
\]
One can check that $a_0$ is a well-defined right $(J,\mu^J)$-invariant function in $\Pi$. It follows that the assignment $f_0\mapsto a_0$ defines a morphism of $G(F)$-modules $\Phi:
  \sW\ra \Pi$. 

\begin{proposition} \label{p:drinfeldMorphism} $\Phi$ is an isomorphism.
\end{proposition} 

\begin{proof} According to \cite[Theorem 7.5]{Roche98}, $(J,\mu^J)$ is a
  cover of $(T(\cO), \bmu)$, in the sense of \cite[Definition
  8.1]{Bushnell98} (in \cite{Roche98}, the residue characteristic is further
 restricted
so as to obtain a nondegenerate bilinear form on the Lie algebra,
but this restriction can be relaxed using the dual Lie algebra as in \cite{Yu-ctsr}:
see \cite[\S 3.1.2, \S A.2]{geometrization}).
  It follows from \cite[Proposition 8.5]{Bushnell98}
  that $(J,\mu)$ is also a cover of $(I_L,\mu^L)$. The explicit
  isomorphism above is constructed (in the general setting of types)
  in \cite[\S 2]{Dat99} (see also \cite[Theorem 2]{Blondel}, where an even more general statement about covers is proved).
\end{proof}

Recall from \eqref{e:sv-defn} that $\sV:=\ind_K^{G(F)} \mu$. By
definition, this is a submodule of $\sW$. On the other hand, let
$P^\circ = L(\cO) U_P^+(F)$. The character $\mu^L$ extends uniquely to
a character of $P^\circ$ which is trivial on $U_P^+(F)$. By an abuse
of notation, we denote this character by $\mu^L$ as well. Then,
recalling the definition of $\Theta$ in \eqref{eq:Theta}, we have an
isomorphism
\[
\Theta := \iota_{P(F)}^{G(F)} \ind_{L(\cO)}^{L(F)} \mu^L \simeq
\ind_{P^\circ}^{G(F)} \mu^L.
\]
We identify $\Theta$ with the $G(F)$-module on the RHS of the above
isomorphism. With this convention, it is clear that we have an
inclusion $\Theta \inj \Pi$. To establish Theorem \ref{t:main2}, we
prove that the restriction of $\Phi$ to $\sV$ defines an isomorphism
$G(F)$-modules $\sV \isom \Theta$. To this end, we define averaging
(or symmetrization) maps $\sW\ra \sV$ and $\Pi\ra \Theta$ and show
that they are compatible with $\Phi$. 
 
Recall that $\sW$ is realized on the space of left
$(J,\mu^J)$-invariant functions on $G(F)$. Under this identification,
the subspace $\sV\subset \sW$ is identified with the space of left
$(L(\cO), \mu^L)$-invariant functions in $\sW$. On the other hand,
$\Pi$ is identified with the space of left $(P^\circ_I,
\mu^I)$-invariant functions on $G(F)$, and $\Theta$ is the subspace of
left $(L(\cO), \mu^L)$-invariant functions in $\Pi$.
 
Choose a Haar measure on $L$ such that the volume of $L(\cO)$ equals
$1$.  For every function $f:G(F)\ra \bC$, define $f^c$ by
 \begin{equation} 
 f^c (x) = \int_{L(\cO)} \mu^L (l) f(l^{-1}x) dl 
\end{equation} 
Then $f\mapsto f^c$ defines a splitting of the natural inclusion of
left $(L(\cO), \mu^L)$-invariant functions on $G(F)$ into the space of all
functions on $G(F)$.  Note that this splitting obviously commutes with
the action of $G(F)$ on the space of all functions by right
translation. Therefore, the map $f\mapsto f^c$ defines a splitting of
the natural inclusions of $G(F)$-modules $\sV\inj \sW$ and $\Theta\inj
\Pi$. Our goal is to show that the diagram 
\begin{equation} \label{e:commsq}
\xymatrix@R=1.0cm@C=1.0cm{ 
\sW \ar@{.>}@/^/[d] \ar[r]^\Phi &  \Pi  \ar@{.>}@/^/[d]\\
\sV \ar@{^{(}->}[u] \ar[r]  & \Theta \ar@{^{(}->}[u]
}
\end{equation} 
commutes, where the averaging maps are denoted by the dotted arrows.  The key computation is the following: 

\begin{proposition} $\Phi(f_0^c)=a_0^c$. 
\end{proposition} 
\begin{proof} By definition, $f_0^c$ is the left $L(\cO)$-symmetrization of $f_0$, $f_0^c =|K/J|^{-1} \mu \cdot \charr(K)$.  This is, however, also the right
symmetrization of $f_0^c$.  Since $\Phi$ commutes with the right action of $G(F)$ (it is a morphism of representations), $\Phi(f_0^c)$ is also the
right $L(\cO)$-symmetrization of $a_0$. This, in turn, is the function supported
on $P^0_I K$ which sends $pk$ to $|K/J|^{-1} \mu^I(p) \mu(k)$ for $p \in P^0_I$ and $k \in K$.  

On the other hand, $a_0^c$ is the left $L(\cO)$-symmetrization of $a_0$, i.e.,
the function supported on $P^0 J$ sending $pj$ to $|P^0/P^0_I|^{-1}\mu^L(p)\mu^J(j)$. Since $|P^0/P^0_I|=|K/J|$, this also coincides with the right-symmetrization, i.e., with $\Phi(f_0^c)$.  
\end{proof}

We now resume the proof of Theorem \ref{t:main2}. 
Note that every element of $\sW$ can be written as a finite linear
combination of elements of the form $f_0.g$, where $g\in G(F)$. Next,
the morphisms $\Phi$ and $f\mapsto f^c$ (which represent morphisms
$\sW\ra \sV$ and $\Pi \ra \Theta$) are $G(F)$-equivariant. Therefore,
the above proposition implies that for all $f\in \sW$, we have
\begin{equation} \Phi(f^c)=\Phi(f)^c.
\end{equation}  
Now given $v\in \sV$, we can write $v=w^c$ for some $w\in
W$. Therefore, $\Phi(v)=\Phi(w^c)=\Phi(w)^c \in \Theta$. Therefore,
$\Phi|_\sV$ defines a morphism of $G(F)$-modules $\sV\ra
\Theta$. This clearly creates the commutative square \eqref{e:commsq}.
Since the dotted arrows are surjective and the top horizontal arrow is
an isomorphism, $\Phi|_{\sV}: \sV \ra \Theta$ is surjective. Since it
is the restriction of $\Phi$, which is an isomorphism, it is also
injective. Thus it is an isomorphism. \qed

\subsection{Proof of Theorem \ref{t:Satake}}
We will continue with the notation of the previous subsection. 

\begin{proposition} We have a canonical isomorphism $\End_{G(F)}
  (\Theta) \simeq \sH(L(F), L(\cO), \mu^L)$.
\end{proposition}

\begin{proof} The fact that $\bmu$ is parabolic with Levi $L$ means
  that $W_\bmu = N_G(\bmu)/T=N_L(T)/ T$. In particular,
  $N_G(\bmu)\subset L(F)$. By the main theorem of \cite{Roche02},
  parabolic induction with respect to $P$ defines an equivalence of
  categories between Bernstein block of $L$ corresponding to the pair
  $(T(\cO), \bmu)$ and that of $G$. Under this equivalence, the
  $L(F)$-module $\ind_{L(\cO)}^{L(F)} \mu^L$ is mapped to $\Theta$.
  Thus, we obtain a canonical isomorphism
  $\End_{G(F)}(\Theta)\simeq \End_{L(F)} (\ind_{L(\cO)}^{L(F)} \mu^L)
  \simeq \sH(L(F), L(\cO), \mu^L)$.
\end{proof}

Note that the algebra $\sH(G(F), K,\mu)$ acts by convolution on the module
$\sV=\ind_{K}^{G(F)} \mu$. It is a standard fact that
$\sH(G(F), K, \mu)\simeq \End_{G(F)} (\sV)$. By Theorem \ref{t:main2},
$\sV$ is canonically isomorphic to $\Theta=\iota_{P}^{G}
\left(\ind_{L(\cO)}^{L(F)} \mu^L\right)$. By the preceding paragraph,
the endomorphism ring of $\Theta$ is canonically isomorphic to the
endomorphism ring of the $L(F)$-module $\ind_{L(\cO)}^{L(F)}
\mu^L$. Therefore, we obtain a canonical isomorphism
$\sH(G(F),K,\mu)\simeq \sH(L(F), L(\cO),\mu^L)$.

Finally, recall that $\mu^L = \mu^{L(F)}|_{L(\cO)}$, where $\mu^{L(F)}: L(F)\ra \bCt$ is a character of $L(F)$.
Then multiplication by $\mu^{L(F)}$
defines a canonical isomorphism of algebras $\sH(L(F),L(\cO)) \simeq
\sH(L(F), L(\cO), \mu^L)$. Moreover, by the Satake isomorphism, we have a
canonical isomorphism $\sH(L(F),L(\cO)) \simeq
\bC[\hT/W_L]=\bC[\hT/W_\bmu]$. Theorem \ref{t:Satake} is established.  \qed
 
\bibliographystyle{alpha}

\bibliography{ref.geometrization}

\end{document}